\documentclass[11pt]{article}
\usepackage{mathrsfs}
\usepackage{amsfonts}

\usepackage[leqno]{amsmath}
\usepackage{graphicx}
\usepackage{latexsym}
\usepackage{amsmath,amsfonts,amssymb,amsthm,mathrsfs,euscript,makeidx,color,pstricks}
\usepackage{enumerate}

\usepackage[colorlinks,linkcolor=blue,anchorcolor=green,citecolor=red]{hyperref}

\usepackage[body={6.5in,9.0in},left=1in,top=1in,centering]{geometry}

\def\sqr#1#2{{\vcenter{\vbox{\hrule height.#2pt
				\hbox{\vrule width.#2pt height#1pt \kern#1pt \vrule width.#2pt}
				\hrule height.#2pt}}}}
\def\5n{\negthinspace \negthinspace \negthinspace \negthinspace \negthinspace }
\def\4n{\negthinspace \negthinspace \negthinspace \negthinspace }
\def\3n{\negthinspace \negthinspace \negthinspace }
\def\2n{\negthinspace \negthinspace }
\def\1n{\negthinspace }

\def\dbE{\mathbb{E}}
\def\dbF{\mathbb{F}}

\def\dbJ{\mathbb{J}}

\def\dbR{\mathbb{R}}
\def\dbS{\mathbb{S}}

\def\dbV{\mathbb{V}}

\def\sL{\mathscr{L}}
\def\sM{\mathscr{M}}

\def\sP{\mathscr{P}}

\def\sU{\mathscr{U}}

\def\sX{\mathscr{X}}

\def\cA{{\cal A}}

\def\cF{{\cal F}}

\def\cT{{\cal T}}

\def\BE{{\bf E}}
\def\BF{{\bf F}}

\def\BH{{\bf H}}

\def\BP{{\bf P}}

\newcommand{\R}{\mathbb R}

\def\Ba{{\bf a}}
\def\Bb{{\bf b}}

\def\Bf{{\bf f}}
\def\Bg{{\bf g}}

\def\ds{\displaystyle}

\def\ns{\noalign{\ss}}

\def\ss{\smallskip}

\def\q{\quad}
\def\qq{\qquad}

\def\({\Big (}
\def\){\Big )}
\def\[{\Big[}
\def\]{\Big]}

\def\lan{\langle}
\def\ran{\rangle}

\def\b{\beta}

\def\d{\delta}
\def\e{\varepsilon}

\def\k{\kappa}

\def\t{\tau}
\def\f{\varphi}

\def\G{\Gamma}

\def\esssup{\mathop{\rm esssup}}

\def\cd{\cdot}

\def\les{\leqslant}

\def\bde{\begin{definition}\label}
	\def\ede{\end{definition}}
\def\be{\begin{equation}}
\def\bel{\begin{equation}\label}
\def\ee{\end{equation}}
\def\bt{\begin{theorem}\label}
	\def\et{\end{theorem}}
\def\bc{\begin{corollary}\label}
	\def\ec{\end{corollary}}
\def\bl{\begin{lemma}\label}
	\def\el{\end{lemma}}
\def\bp{\begin{proposition}\label}
	\def\ep{\end{proposition}}
\def\bas{\begin{assumption}\label}
	\def\eas{\end{assumption}}
\def\br{\begin{remark}\label}
	\def\er{\end{remark}}
\def\bex{\begin{example}\label}
	\def\ex{\end{example}}
\def\ba{\begin{array}}
	\def\ea{\end{array}}
\def\ben{\begin{enumerate}}
	\def\een{\end{enumerate}}

\def\square#1{\vbox{\hrule\hbox{\vrule height#1%
			\kern#1\vrule}\hrule}}
\def\rectangle#1#2{\vbox{\hrule\hbox{\vrule height#1%
			\kern#2\vrule}\hrule}}

\font\tenbb=msbm10 \font\sevenbb=msbm7 \font\fivebb=msbm5

\newfam\bbfam
\scriptscriptfont\bbfam=\fivebb \textfont\bbfam=\tenbb
\scriptfont\bbfam=\sevenbb

\newtheorem{theorem}{\indent Theorem}[section]

\newtheorem{proposition}[theorem]{\indent Proposition}
\newtheorem{corollary}[theorem]{\indent Corollary}
\newtheorem{lemma}[theorem]{\indent Lemma}
\theoremstyle{definition}
\newtheorem{definition}[theorem]{\indent Definition}
\newtheorem{remark}[theorem]{\indent Remark}
\newtheorem{example}[theorem]{\indent Example}

\newtheorem{assumption}[theorem]{\indent Assumption}

\makeatletter

\@addtoreset{equation}{section}
\makeatother

\def\bea{\begin{equation*}}
\def\eea{\end{equation*}}

\def \bel{\begin{equation}\label}
\def\eel{\end{equation}}

\def\ba{\begin{array}}
\def\ea{\end{array}}
\newcommand{\ad}{&\!\!\!\displaystyle}

\allowdisplaybreaks

\def\({\Big (}
\def\){\Big )}
\def\[{\Big[}
\def\]{\Big]}
\def\q{\quad}
\def\qq{\qquad}

\def\d{\delta}
\def\e{\varepsilon}

\def\ds{\displaystyle}

\def\ns{\noalign{\smallskip}}

\def\argmin{\mathop{\rm argmin}}
\begin{document}

\title{ Closed-loop Equilibrium for Time-Inconsistent  McKean-Vlasov Controlled Problem\thanks{This research was supported in part by the Simons Foundation (grant award number 523736).}}

\author{Hongwei Mei\thanks{Department of Mathematics, The University of Kansas, Lawrence, KS 66045, U.S. (hongwei.mei@ku.edu).}\,\, \ \ and \ \ Chao Zhu\thanks{Department of Mathematical Sciences, University of Wisconsin-Milwaukee, Milwaukee, WI 53201, U.S. (zhu@uwm.edu).}}
\maketitle
\begin{abstract} The paper  deals with   a class of time-inconsistent control problems for   McKean-Vlasov dynamics.  By solving a backward time-inconsistent Hamilton-Jacobi-Bellman  (HJB for short) equation coupled with a forward distribution-dependent stochastic differential equation,  we investigate the existence and uniqueness  of a closed-loop equilibrium for such time-inconsistent distribution-dependent control problem. Moreover,    a special case of semi-linear   McKean-Vlasov dynamics  with a quadratic-type cost functional is considered due to its special structure.
\end{abstract}
\section{Introduction}

Let  $(\Omega,\cF,\dbF, \BP)$ be a complete filtered probability space on which is defined an $m$-dimensional standard Brownian distribution $\{W(t):0\leq t\leq T\}$, where  $\dbF=\{\cF_t:0\leq t\leq T\}$   is the natural  filtration augmented by all $\BP$-null sets. Denote by $\sP_2(\dbR^d)$  the probability measure space on $\dbR^d$ with finite second moments.

The general controlled  McKean-Vlasov dynamic can be formulated as the following  stochastic differential equation (SDE for short)
\begin{equation}\label{intrequ}\begin{cases} dX(t)=a(t,X(t),\rho(t);u(t))dt+b(t,X(t),\rho(t);u(t))dW(t);\q 0\leq t\leq T;\\
\ns X(0)=\xi,\end{cases} \end{equation}
where $t_{0}\in [0,T)$,  $\rho(t)$ is the distribution law of $X(t)$, $a:[0,T]\times\dbR^d\times\sP_2(\dbR^d)\times U\mapsto\dbR^d$ and $b:[0,T]\times\dbR^d\times\sP_2(\dbR^d)\times U\mapsto\dbR^{d\times m}$. In \eqref{intrequ}, $X(t)$ is called the {\it state-process} valued in $\dbR^d$ and $u(t)$ is called the {\it control-process}
valued in a  metric space $U$. 
In some literatures, \eqref{intrequ} is also called a {\it distribution-dependent controlled diffusion.} Under some mild conditions, given $u(\cdot)$ in the space of admissible controls $\sU$, it can be proved that \eqref{intrequ}  possesses a unique solution in some appropriate space.

Let the running cost function  $ f_0:[0,T]\times\dbR^d\times U\mapsto \dbR^+$ and the terminal cost function  $ g_0:\dbR^d\mapsto \dbR^+$ be measurable. The cost functional  is defined as
\begin{equation}\label{intrvalue}J^0(t_0,\xi;u(\cdot))=\BE\bigg[\int_{t_0}^Te^{-\lambda(s-t_0)} f_0\big(s,X(s);u(s)\big)ds+e^{-\lambda(T-t_0)} g_0\big(X(T)\big)\bigg].\end{equation}
where $\lambda>0$ is a {\it discounting rate}. A common optimization problem is  to find an admissible $u^*(\cdot)$ such that
$$J^0(t_0,\xi;u^*(\cdot))=\inf_{u(\cdot)\in\sU}J^0(t_0,\xi;u(\cdot)).$$

If the system \eqref{intrequ} is independent of distribution law (i.e., $a,b$ are independent of $\rho$), such a problem reduces to  the classical stochastic  control problem  which  has been well investigated in the last century (e.g. see \cite{Yong1999}). One of the well-known approaches is to derive  a HJB  equation  through dynamic programming. 

If the system is distribution-dependent, i.e. McKean-Vlasov dynamics, the problem becomes different.  The analysis of McKean-Vlasov SDEs has a long history since the pioneering work  \cite{Kac1956,Mckean1967} and has   attracted  resurgent   attentions in recent years thanks to the recent developments in mean-field game (MFG) problems.  Compared with the classical   optimal control problems for Markov processes, the counterpart  for McKean-Vlasov processes becomes  different because   the dynamic programming principle can not be applied directly. 
 A common approach to overcome such a difficulty is to lift the state space up into the space of probability  measures. Through the Bellman principle, one can derive   an HJB equation on the space of probability measures (e.g. see \cite{Pham2018} and the references therein). Using mass transport theory, it is possible to investigate the viscosity solution of the HJB equation on the space of probability measures. See also \cite{Yong2015}, which treats   a special  linear-quadratic case and obtains an optimal feedback control   by analyzing a linear mean-field forward-backward 
 SDE derived from a variational method. 

An alternative  idea  
is developed in a serial papers on    MFG \cite{Larsy2007a,Larsy2007b, Larsy2006a,Larsy2006b}
 and \cite{
 	Huang2007,Huang2007a,Huang2007b,Huang2010b} where the authors considered a backward HJB equation coupled with a forward  transport equation on the space of probability measures (or a forward SDE) to derive a mean-field equilibrium. Applying the similar idea to the explicit  control problem of system \eqref{intrequ} with cost functional \eqref{intrvalue},
the first step is to solve a classical HJB equation  which is concluded from the classical optimal control  of system  given a guiding (fixed) process $\rho(\cdot)$. At the same time, a  feedback control can be determined if the HJB equation is regular enough. The second step is to verify that  the guiding process $\rho(\cdot)$ coincides with the distribution law of the solution process of SDE \eqref{intrequ} using the feedback control. If the two-step verification is fulfilled, the feedback control is called a mean-field equilibrium. One can see that the mean-field equilibrium is essentially defined by a fixed point process. Note, however,   such an equilibrium  is not an optimal control strategy in general.    

Given a guiding process $\rho(\cdot)$, the HJB equation in the MFG is derived from  the dynamic programming and Bellman principle if the cost function in \eqref{intrvalue} is  exponential discounting. This is the so-called time-consistent case, i.e.,  the optimal control determined now stays optimal in the future.  
In many real life problems, such a requirement is too ideal and far from reality.  For example, one must often adjust her decisions as time goes by. Mathematically,  if the cost function is in  non-exponential discounting or hyperbolic discounting situations,  the control problem is not time-consistent anymore. For those problems, it is impossible for us to find an optimal control at the initial time which stays optimal in the future. This is the so-called time-inconsistency. The main idea is to find a local optimal control or strategy (instead of a  global optimal) to save for future.  Lots of works have been devoted to dealing with time-inconsistency in the last decade (e.g. see  \cite{Hu2010,Ekeland2010,Ekeland2008,Ekeland2012,Qi2017,Yong2012a,Yong2012b,MeiYong2019,Wei2017}). One also may refer to the survey paper \cite{Yan2019} and the references therein. Among those papers, two types of time-inconsistent equilibrium are considered, namely the open-loop  equilibrium control and the closed-loop equilibrium strategy. For example,  the open-loop equilibrium control for linear-quadratic case is characterized via a
maximum-principle-like methodology  in \cite{Hu2010}. To consider the closed-loop
equilibrium strategy for time-inconsistent control problem, the author derived a so-called time-inconsistent HJB equation via an $N$-player game in  \cite{Yong2012b} and verifies the local optimality in \cite{Wei2017}.

 Compared to time-consistent problems, time-inconsistency brings new interesting features  as well as  mathematical challenges. 
 One of the main    difficulties  brought by time-inconsistency for general diffusions in $\dbR^d$  lies in  the existence of time-inconsistent equilibrium strategies. For non-degenerate stochastic diffusions  in $\dbR^d$, the  existence and uniqueness  of (closed-loop) time-inconsistent equilibrium  can be found in \cite{Yong2012b}.   While for degenerate case, the existence  is  still an open problem due to  the lack of first-order regularity of the viscosity solution for a degenerate  second-order HJB equation. More explicitly, for a time-inconsistent problem in the space of $\dbR^d$, the identification of time-inconsistent equilibrium  requires that the  HJB equation admits a classical  solution, which is not necessarily true for a degenerate problem. Thus in this paper, we will assume that the system is degenerate for general cases. For the special  semi-linear-quadratic case, since the HJB has an explicit form of solutions, the non-degeneracy assumption is not necessary anymore.

In this paper, we are devoted to proving the existence and uniqueness of the {\it  equilibrium} (see Definition \ref{def-optimalcurve}) for a class of controlled McKean-Vlasov dynamics with a time-inconsistent cost functional. Previous works on time-inconsistent distribution-dependent diffusions include \cite{Ni2018a,Ni2018b,Wang2018,Yong2015} for example. We note that  the aforementioned  papers are mainly focused on a special case of linear-quadratic problems. In \cite{MeiY2019}, the authors deal with time-inconsistent distribution-dependent control problems for finite-stated Markov chains. Different from the aforementioned works, in this paper we deal with  a class of time-inconsistent control problems for   McKean-Vlasov dynamics in $\dbR^d$.

The paper is arranged as follows.
Some frequently used notations as well as some preliminary results  will be introduced in Section \ref{sect-notations}. In Section \ref{sec-mdcd}, we will introduce our main system and prove the existence and uniqueness of the solution to our system. In  Section \ref{sec:mic}, we will review the main results for time-inconsistent distribution-independent control problems. In Section \ref{sec-equi}, we will present the definition of time-inconsistent distribution-dependent equilibrium and prove its existence and uniqueness. In Section \ref{sec-semi-linear}, we will consider the similar results for a class of semi-linear systems with a quadratic cost. In Section \ref{sec:mfg}, we set a mean-field game whose equilibrium coincides with the equilibrium we find in our control problem. Finally   some concluding remarks will be made in Section \ref{sec:conrem}.

\subsection{Notations and Preliminaries}\label{sect-notations}

Let
$\sL^2(\dbR^d)$ the collection of $\dbR^d$-valued random variables with a finite second moment, i.e.,
$$\sL^2(\dbR^d):=\{X:\Omega\mapsto\dbR^d\big|X \text{ is $\cF$-measurable with }\BE|X|^2<\infty\}.$$
$\sL^2(\dbR^d) $ is equipped with the  norm
$\Vert X\Vert_{\sL^2}:=(\BE\vert X\vert^2)^{\frac12}.$
For any $X\in \sL^2(\dbR^d)$, denote by $\text{law}(X)$  the distribution of $X$.

Let $\sP_2(\dbR^d)$ be the space of  probability measures  with finite second moments equipped with Wasserstein-$2$ metric $w(\cdot,\cdot)$, i.e.
$$\ba{ll}w^2(\rho,\gamma)\ad:=\inf_{\pi\in\Pi^{\rho,\gamma}}\int_{\dbR^d\times\dbR^d}|x-y|^2\pi(dx,dy),
\ea$$
where $$\Pi^{\rho,\gamma}:=\{\pi\in \sP(\dbR^d\times\dbR^d):\pi(dx,\dbR^d)=\rho(dx),\pi(\dbR^d,dy)=\gamma(dy)\}.$$
It is easy to see that
\begin{equation}\label{cwE}w^2(\text{law}(X),\text{law}(Y))\leq \Vert X-Y\Vert^2_{\sL^2}.\end{equation}
We refer to  \cite{Villani2008} for more discussions  on Wasserstein  metrics. 

Let $\sP_2^{\d,C}(\dbR^d)$ be a  subset of $\sP_2(\dbR^d)$ defined by
$$\sP_2^{\d,C} (\dbR^d)=\bigg\{\rho\in \sP_2(\dbR^d):\int_{\dbR^d}|x|^{2+\d}\rho(dx)\leq C\bigg \}.$$
Note that   $\sP_2^{\d,C}(\dbR^d)$  is a compact subset of  $(\sP_2(\dbR^d),w)$ for any $\d,C>0$. 

	

Throughout the paper, we suppose that $\cF_0$ is large enough such that for any $\rho\in\sP_2(\dbR^2)$, there exists a $\xi\in \cF_0$ such that $\text{law}(\xi)=\rho$.

%
%
%
%
Let $\sX:=L_{\dbF}^2(\Omega;C([0,T],\dbR^d))$ be   defined as 
 $$\ba{ll}\sX\ad:=\Big\{ X:[0,T]\times\Omega\mapsto\dbR^d:  X\text{ is $\dbF$-progressively measurable}
     \\\ns\ad\qq\qq\qquad  \text{ and continuous with  }
 \BE\sup_{0\leq t\leq T}|X(t)|^2<\infty\Big\}.\ea$$
Let $ \sM:=C([0,T],\sP_2(\dbR^d))$ be the set of $\sP_2(\dbR^d)$-valued  continuous curves on $[0,T]$ equipped with the uniform metric $m$, i.e.,
\begin{equation}
\label{eq-m-metric}
m(\mu_1,\mu_2):=\sup_{0\leq t\leq T}w(\mu_1(t),\mu_2(t)).
\end{equation}
Since $(\sP_2(\dbR^d),w)$ is complete,  so is $ (\sM, m)$. 

 Write $\sM_\gamma:=\{\mu\in\sM:\mu(0)=\gamma\}.$ Define a subset of $\sM_\gamma$ by $$\ba{ll}\sM^{\d,C,\lambda}_{\gamma}\ad:=\bigg\{\mu\in\sM_{\gamma}: 
 \sup_{0\leq t\neq s\leq T}\frac{w^2(\mu(t),\mu(s))}{|t-s|}\leq \lambda,\  \mu(t)\in \sP_2^{\d, C} (\dbR^d) \text{ for any }t\in[0,T] \bigg\}.\ea $$
By the well-known Arzela-Ascoli lemma, $\sM^{\d,\lambda}_{\gamma}$ is a compact and convex  subset of $\sM$.

For any $\mu\in\sM$, we write 
$$\text{Graph}(\mu):=\{(t,\mu(t))\in[0,T]\times\sP_2(\dbR^d):0\leq t\leq T\}.$$

By \eqref{cwE},  for any $X\in\sX$, $\text{law}(X(t))$ is  continuous with respect to $t$ under the $w$-metric. Thus
we  define a map $\text{LAW}: \sX\mapsto  \sM$
by
$$\text{LAW}(X)(t):=\text{law}(X(t)),\q\text{for }0\leq t\leq T.$$
We have the following lemma.
\begin{lemma} \label{lemmaXrho}
	If $X_n\in \sX$ satisfies
	$$\lim_{n,m\rightarrow\infty}\sup_{0\leq t\leq T}\BE\vert X_n(t)-  X_m(t)\vert^2=0,$$ then
	$\text{\rm LAW}(X_n)\in \sM$ is a Cauchy sequence.
\end{lemma}
\begin{proof}
	It is a direct consequence of \eqref{cwE}. 
\end{proof}

Finally let the control space $U$ be a metric space equipped with metric $d_U(\cdot,\cdot)$. Also let $v_0$ be some fixed point in $U$.

\section{Distribution-dependent Time-inconsistent Control}\label{sec-mdcd}

On the complete probability space $(\Omega,\cF,\BP)$, we consider the following distribution-dependent controlled stochastic differential equation
\begin{equation}\label{gen-SDE}\begin{cases} dX(t)= a(t,X(t),\rho(t);u(t))dt+ b(t,X(t),\rho(t))dW(t),\\
\rho(t)=\text{law}(X(t)),\q X(t_0)=\xi\in \cF_{t_0}\end{cases}\end{equation}
where $a:[0,T]\times\dbR^d\times\sP_2(\dbR^d)\times U\mapsto \dbR^d$, $b:[0,T]\times\dbR^d\times\sP_2(\dbR^d)\mapsto \dbR^{d\times m}$,   and $u(t)$ is the control process valued in $U$. From now on, we only consider the case when $b$ is independent of $u$, the reason of which will be explained later.

Define 
$L^2_\dbF([t_0,t_1],U)$ by
$$\begin{aligned}L^2_\dbF([t_0,t_1],U)=\bigg\{u(\cdot):[t_0,t_1]\mapsto U: u(\cdot) \text{ is $\dbF$-progressively measurable}\\\text{ with } \BE\int_{t_0}^{t_1} d^2_U(u(t),v_0)dt<\infty\bigg\}. \end{aligned}$$
Since  we are concerned with closed-loop strategies in this paper, we 
write
\begin{equation}\label{admissiblecontrol-2}\sU_\alpha:=\{{ u}:[0,T]\times\dbR^d\mapsto U\big |u(t,\cdot) \text{ is Lipschitz with a uniform Lipschitz constant $\alpha$} \}.\end{equation}
The space of admissible closed-loop Lipschitz strategy is defined as
\begin{equation}\label{admissiblecontrol}\sU:=\bigcup_{\alpha>0}\sU_\alpha.\end{equation}
Under some mild conditions (e.g. Lipschitz conditions for $b$ and $\sigma$), one can easily see that if $ u\in\sU$, the feed back control process $u(\cdot,X(\cdot))\in L^2_\dbF([0,T],U)$. Thus in essence we can regard $\sU$ as a subset of $L^2_\dbF([0,T],U)$. 

Given the running cost $f:[0,T]\times [0,T]\times\dbR^d\times U\times\sP_2(\dbR^d)\mapsto\dbR $ and the terminal cost $g: [0,T]\times\dbR^d\times U\times\sP_2(\dbR^d):\mapsto\dbR$, the  cost functional under strategy $u\in L^2_\dbF([0,T],U)$ is defined as
$$\dbV(t_0,\xi;u):=\dbJ(t_0;t_0,\xi; u)$$
where
\begin{equation}\label{cost}\dbJ(\t;t_0,\xi;u):=\BE_{t_0,\xi}\bigg[\int_{t_0}^T f\big(\t;s,X(s),\text{law}(X(s));u(s)\big)ds+ g\big(\t;X(T),\text{law}(X(T))\big)\bigg],\q \xi\in \cF_{t_0}.\end{equation}

If we restrict $u\in\sU$ and let $X$ be the solution of \eqref{gen-SDE}, then the value of $\dbJ$ and hence $\dbV$   depend only on the distribution of $\xi$.   Thus 
 we may write
$\dbJ(\t;t,\text{law}(\xi);u)$ instead of $\dbJ(\t;t,\xi;u)$. Our main effort of paper is to derive a   closed-loop strategy (see Definition \ref{def-optimalcurve}) for such a time-inconsistent distribution-dependent problem.  From now on we only consider the case $u\in\sU$, i.e., closed-loop strategy.

\begin{remark}\label{remconfu}{\rm (1)
		 One may question why the feedback control $u$ is only a function of $(t,x)$ (not distribution-dependent).  Essentially, we 
		 incorporate  the dependence on the distribution  into the dependence of $t$.
		 
		 (2) Since our strategy is in a closed-loop form which is derived from a time-inconsistent HJB equation,  we assume that the diffusion coefficient $b$ is independent of the control  $u$ to avoid the analysis on the second order regularity of the time-inconsistent HJB equation.
		 
		 (3) Note that our time-inconsistent cost is distribution-independent. Let's see the following simple example. 
	Suppose $d=1$  and that the terminal cost in \eqref{cost} is distribution dependent in the following form
	$$g_\t(x,\rho)=h(\t)\(x-\int_{\dbR^d}x\rho(dx)\)^2.$$
	Then
	$\int_{\dbR^d}g_\t(x,\rho)\mu(T,dx)$ is the product of $h(\t)$ and the variance of the distribution  $\mu(T)$. In this case, if we let
	$$\bar g_\t(x,\rho)=h(\t)\(x^2-\(\int_{\dbR^d}x\rho(dx)\)^2\),$$ then
	$\int_{\dbR^d}\bar g_\t(x,\rho)\mu(T,dx)=\int_{\dbR^d} g_\t(x,\rho)\mu(T,dx)$. Thus  $g$ and $\bar g$ give the same terminal functional in $\dbV$. We will see that in the process of deriving the fixed-point, the  terminal conditions in the Hamilton-Jacobi equation will be different. As a consequence, the time-inconsistent equilibrium will be different as well. Therefore, to avoid such possible confusion, we will compare our problem with a time-inconsistent mean-field for infinite symmetric players. Roughly speaking, the forms of $f$ and $g$ are determined by the model. 
		
	}
\end{remark}

\subsection{Existence and Uniqueness of the Solution}
In this subsection, we   show  that under 
Assumption \ref{A-0}, \eqref{gen-SDE} admits a unique solution for any $u\in\sU$.  

\begin{assumption}\label{A-0} 
	$a:[0,T]\times \dbR^d\times\sP_2(\dbR^d)\times U\mapsto\dbR^d$ and $b:[0,T]\times \dbR^d\times\sP_2(\dbR^d)\mapsto\dbR^{d\times m}$ are continuous and   satisfy 
	$$\begin{cases}
	\ad|a(t,0,\rho;v_0)|+|b(t,0,\rho)|\leq \k_0\big(1+\int_{\dbR^d}|x|^2\rho(dx)\big),\\
	\ns\ad
	 |a(t,x_1,\rho_1;v_1)-a(t,x_2,\rho_2;v_2)|\leq \k_0\big(|x_1-x_2|+d_U(v_1,v_2)+w(\rho_1,\rho_2)\big),\\ 
	\ns\ad|b(t,x_1,\rho_1)-b(t,x_2,\rho_2)|\leq \k_0(|x_1-x_2|+w(\rho_1,\rho_2)),
	%
	\end{cases}$$ for all $t\in [0,T]$, $x_{1},x_{2} \in \R^{d}$, $\rho_{1},\rho_{2}\in \sP_2(\dbR^d) $, and $v_{1},v_{2}\in U$.
\end{assumption}
\begin{lemma} \label{lawXM} Under Assumption \ref{A-0}, the the following assertions hold:
	\begin{itemize}
  \item[\rm (1)] For any $ u\in\sU$ and $\xi\in\cF_{0}$ with $\text{\rm law}(\xi)\in\sP_2(\dbR^d)$, there exists a unique solution $X\in L_{\dbF}^2(\Omega;C([0,T],\dbR^d))$ to \eqref{gen-SDE}. As a result,  by Lemma \ref{lemmaXrho}, $\text{\rm LAW}( X)\in\sM$.
	
	\item [\rm(2)] For any $u\in\sU_\alpha$, there exists a constant $\beta_\alpha$ independent of $u$ such that  
\bel{holder} \BE\sup_{0\leq t\leq T}|X(t)|^2\leq\beta_\alpha(1+\BE|\xi|^2)\q\text{and}\q w^2(\text{\rm law}(X(t)),\text{\rm law}(X(s)))\leq  \beta_\alpha(1+\BE|\xi|^2)|s-t|.\eel
	
	  \item[\rm(3)] If  $\gamma:=\text{\rm law}(\xi)$ has a finite $(2+\d)$th moment, then for any $u\in \sU_\alpha$,  there exist constants  $C_{\alpha,\gamma,\d}, \lambda_{\alpha,\gamma}>0$ such that
	\bel{Xcompact}\text{\rm LAW}(X)\in \sM_\gamma^{\d,C_{\alpha,\gamma,\d},\lambda_{\alpha, \gamma}}\eel
	where $C_{\alpha,\gamma,\d}$ depends on  $\alpha$, $\d$ and the $(2+\d)$th moment of $\gamma$ only and $\lambda$ depends on $\alpha$  and the second moment of $\gamma$ only.

\end{itemize}
\end{lemma}

\begin{proof} (1) The proof is a   direct application of Picard's iteration and similar to that of Theorem 1.7 of \cite{Carmona-16}. We present the proof here  for reader's convenience.  

Let $X_0(t)=\xi$ for $0\leq t\leq T$. We define $\{ X_n,\mu_n=\text{LAW}(X_n)\}_{n\geq 0}$  recursively by
$$X_{n+1}(t):=\xi+\int_{0}^ta\big(s,X_n(s),\mu_n(s);u(s,X_n(s))\big)ds+\int_{0}^tb(s,X_n(s),\mu_n(s))dw(s),\q0\leq t\leq T.$$
Using It\^o's formula, we have 
\begin{equation}\label{difference}\ba{ll}\ad\BE\sup_{0\leq s\leq t}|X_{n+1}(s)-X_n(s)|^2\\
\ns\ad\leq 2T\int_{0}^t\BE\sup_{0\leq r\leq s}|a(r,X_n(r),\mu_n(r);u(s,X_n(r))\\
\ns\ad\qq\qq\qq\qq\qq\qq-a(r,X_{n-1}(r),\mu_{n-1}(r);u(r,X_{n-1}(r)))|^2ds\\
\ns\ad\q+2\BE\int_{0}^t|b(s,X_n(s),\mu_n(s))-b(s,X_{n-1}(s),\mu_{n-1}(s))|^2ds\\
\ns\ad\leq 2LT\int_{0}^t\BE\sup_{0\leq r\leq s}|X_n(r)-X_{n-1}(r)|^2ds.\ea\end{equation}
 Simple calculation yields that for some positive constant $M$, we have
 $$\BE\sup_{0\leq s\leq T}|X_{n+1}(s)-X_n(s)|^2\leq \frac {(2LT)^n}{n!}M,$$
 and hence
 $$\sum_{n=1}^\infty\BE\sup_{0\leq s\leq T}|X_{n+1}(s)-X_n(s)|^2<\infty.$$
 This concludes that 
 $ X_n$ is a Cauchy sequence in $\sX$ with a limit written as $ X$.  Also denote $\mu(s): = \text{law}(X(s))$ for $s\in [0, T]$.  The following verifies that $ X$ is the solution:
 $$\ba{ll}\ad\BE\sup_{0\leq t\leq T}\left|X(t)-\(\xi+\int_{0}^ta(s,X(s),\mu(s);u(s,X(s))ds+\int_{0}^tb(s,X(s),\mu(s))dw(s)\)\right|^2\\
 \ns\ad \leq \BE\sup_{0\leq t\leq T}|X(t)-X_n(t)|^2\\
 \ns\ad\q+T\BE\int_{0}^T|a(s,X(s),\mu(s);u(s,X(s))-a(s,X_n(s),\mu_n(s);u(s,X_n(s))|^2ds\\
 \ns\ad\q+\BE\int_{0}^T|b(s,X(s),\mu(s))-b(s,X_n(s),\mu_n(s))|^2ds\\
 \ns\ad\rightarrow0,\qq\text{as }n\rightarrow\infty.\ea$$
By virtue  of \eqref{difference},  the solution $X$ is unique and $\text{LAW}(X)\in\sM$. 

(2) Using It\^{o}'s formula and Grownwall's inequality, the standard arguments reveal   that for some constant $C_\alpha>0,$
\bel{ubounsupX}\BE\sup_{0\leq t\leq T}|X(t)|^2\leq C_\alpha(1+\BE|\xi|^2).\eel
Note also that $$\BE|X(t)-X(s)|^2\leq L_\alpha (1+\BE|\xi|^2)|t-s|,$$ which leads to \eqref{holder}  directly.

(3) Similarly, using It\^{o}'s formula,  one can prove that  for some constant $D_{\alpha,\d}>0$, it follows that
\bel{ubounsupX-2}\BE\sup_{0\leq t\leq T}|X(t)|^{2+\d}\leq D_{\alpha,\d}(1+\BE|\xi|^{2+\d}).\eel  \eqref{Xcompact} is a direct conclusion of \eqref{holder} and \eqref{ubounsupX}. The proof is complete.
\end{proof}

Thanks to   Lemma \ref{lawXM}, we can now  define $\cT^\gamma_1:\sU\mapsto\sM_\gamma$ by $$\cT^\gamma_1(u):=\text{LAW}(X),$$ where $X$ is the  distribution curve determined by \eqref{gen-SDE} under the strategy $u$ with initial law $\gamma\in \sP_2(\dbR^d)$.
%
The following lemma shows that  $ \cT^\gamma_1$ depends  continuously  on $ u$  in some sense.
\begin{lemma}\label{contiuityofsolution}
	 For   $u_1,u_2\in\sU_\alpha$ with
	$$d_{U}(u_1(t,x),u_2(t,x))\leq (1+|x|)\e,$$
	 there exist constants $\beta_1,\beta_\alpha>0$ such that
	\begin{equation}\label{Xnconvergence}m(\cT^\gamma_1(u_1),\cT^\gamma_1(u_2))\leq\beta_1\beta_\alpha\e T\(1+\int_{\dbR^d}|x|^2\gamma(dx)\).\end{equation}
\end{lemma}
\begin{proof}  
	For $i=1,2$, let $X_i$ be the solution of \eqref{gen-SDE} under strategy $ u_i$ and $\mu_i$ be the corresponding distribution curve.  Simple calculation yields that
	$$\ba{ll}\ad\sup_{t_0\leq t\leq T'}\BE|X_1(t)-X_2(t)|^2\\
	\ns\ad\q\leq 2T\BE\int_{0}^{T'}|a(t,X_1(t),\mu_1(t);u_1(t,X_1(t)))-a(t,X_2(t),\mu_2(t);u_2(t,X_2(t)))|^2dt\\
	\ns\ad\qq+2\BE\int_{0}^{T'}|b(t,X_1(t),\mu_n(t);u_1(t,X_1(t))-b(t,X_2(t),\mu(t);u_2(t,X_2(t))|^2dt\\
	\ns\ad\q\leq L\int_{0}^{T'}\(\BE|X_1(t)-X_2(t)|^2+(1+\BE|X_2(t)|^2)\e\)dt\\
	\ns\ad\q\leq L\int_{0}^{T'}\sup_{0\leq t\leq s }\BE|X_1(t)-X_2(t)|^2ds+L\e\(1+\sup_{0\leq t\leq {T}}\BE|X_2(t)|^2\).\ea$$
	Grownwall's inequality implies that for some $\beta_1>0$
	$$\sup_{0\leq t\leq T}\BE|X_1(t)-X_2(t)|^2\leq \beta_1T\e(1+\sup_{0\leq t\leq {T}}\BE|X(t)|^2).$$
 The desired assertion \eqref{Xnconvergence} then follows from  \eqref{holder} and the definition of $m$ in \eqref{eq-m-metric}.  \end{proof}

\section{Time-inconsistent Distribution-independent Control}\label{sec:mic}

In this section, we  briefly review the results on the time-inconsistent control problem in \cite{Yong2012b}. We need the following assumption.

\begin{assumption}
	\label{A-1} 
	{\rm (1)} Suppose there exist $a_1:[0,T]\times\dbR^d\times\sP_2(\dbR^d)\mapsto\dbR^d$, $a_2:[0,T]\times\dbR^d\times U\mapsto\dbR^d$,  $f_1:[0,T]\times[0,T]\times\dbR^d\times\sP_2(\dbR^d)\mapsto\dbR^d$, $f_2:[0,T]\times[0,T]\times\dbR^d\times U\mapsto\dbR^d$ such that $$a(t,x,\rho;u)=a_1(t,x,\rho)+a_2(t,x;u) \text{ and }f(\t,t,x,\rho;u)=f_1(\tau,t,x,\rho)+f_2(\tau,t,x;u).$$ 
	
	{\rm (2)}
	 There exists a  map $\psi:[0,T]\times\dbR^d\times\dbR^d\mapsto U$ such that
	\begin{equation}\label{defpsi}\psi(t,x,q):=\argmin_{v\in U}\left\{q\cdot  a_2(t,x;v)+f_2(t;t,x;v)\right\}\end{equation}
	with 
	$$\left\{\ba{ll}\ns\ad d^2_U(\psi(t,x,q),v_0)\leq \beta_\psi(1+|x|^2+|q|^2)\\ [2mm]
	\ad d_U(\psi(t,x_1,q_1),\psi(t,x_2,q_2))\leq \beta_\psi (|x_1-x_2|+|q_1-q_2|),
	\ea\right.$$ where $\beta_{\psi}$ is a positive constant.

\end{assumption}

The following example demonstrates that Assumption \ref{A-1} is not hard to verify in many situations.
\begin{example} 
	Let $ A(\cdot):[0,T]\times\dbR^d\mapsto\dbR^d$, $\varphi(\cdot):\dbR^d\mapsto\dbR^l$, $G(\t;t,\cdot),S(\t;\cdot):\dbR^d\mapsto\dbR$  $ B(\cdot)\in\dbR^{d\times l}, C(\cdot),D(\cdot)\in\dbR^{d\times l}$  and  $R(\cdot;\cdot)\in\dbR^{l\times l}$  be continuous.   Assume that $R\geq\lambda I$ for some $\lambda>0$,  $$a(t,x,\rho,u):= A(t,x)+2B(t)u+C\int_{\dbR^d}\varphi(t,y) \rho(dy)\q\text{and}\q b(t,x,\rho)=D\int_{\dbR^d}\varphi(t,y) \rho(dy)$$
	with 
	$$f(\t;t,x,u)=G(\t;t,x)+\langle u, R(\t;t)u\rangle\q\text{and}\q g(\t;x)=S(\t;x).$$
	Simple calculation yields that
	$$\psi(t,x,q)=-R(t;t)^{-1}B(t)'q.$$
	Then Assumption \ref{A-1} is satisfied.
\end{example}


Now let's present the time-inconsistent HJB equation  for time-inconsistent distribution-independent problems. The reader is referred to   \cite{Yong2012b} on   the derivation of such an equation via   an $N$-player game. 

We consider the following distribution-independent SDE with a priori $\mu\in\sM$,
\begin{equation}\label{SDE-3}dY(t)= a(t,Y(t),\mu(t);u(t))dt+ b(t,Y(t),\mu(t))dw(t)\end{equation}
with cost function 
\begin{equation}\label{cost-3}J^{\mu}(\t;t,y; u):=\BE_{t,y}\(\int_{t}^T f(\t;s,Y(s),\mu(s);u(s,Y(s)))ds+ g(\t;Y(T),\mu(T))\), \text{for }u\in\sU\end{equation}
and  value function
\begin{equation}\label{value-2}V^{\mu}(t,y; u):=J(t;t,y;u).\end{equation}

The  {\it equilibrium strategy}  is defined as
\begin{equation}\label{equistra} u(t,y)=\psi(t,y,D\Theta(t;t,y))\end{equation}
where $\Theta(\tau;t,y)$ is the  solution to the following time-inconsistent HJB equation (given $\mu$) 
\begin{equation}\label{hjb-equi}\left\{\ba{ll}\ad \Theta_t(\t;t,y)+\frac12\text{Tr}[b(t,y,\mu(t))b'(t,y,\mu(t))D^2\Theta(\t;t,y)]\\
\ns\ad~+a\big(t,y,\mu(t);\psi(t,y,D\Theta(t;t,y))\big)\cdot D\Theta(\t;t,y)+f(\tau;t,y,\mu(t);\psi(t,y,D\Theta(t;t,y)))=0;\\
\ns\ad \Theta(\t;t,y)=g(\t;y,\mu(T)).\ea\right.\end{equation}
Here $u$ is independent of $\mu$  due to Assumption \ref{A-1}.

Under some appropriate conditions, it is shown  in \cite{Yong2012b}  that there exists a unique solution of \eqref{hjb-equi} whose first-order  derivative is Lipschitz.   Therefore  the  strategy  $ u\in\sU$ is well defined by \eqref{equistra}.
%
From  the previous arguments,   we  define a map $\cT_2:\sM\mapsto\sU$ by
$$\cT_2(\mu):=u$$
where $ u$ is defined in \eqref{equistra} through solving the HJB equation \eqref{hjb-equi}. From \cite{MeiYong2019,Wei2017}, we know that  $u=\cT_2(\mu)$ verifies the following local optimality  condition: 
\begin{proposition} \label{prooptimal} For   $u= \cT_2(\mu)$, we have 
	$$\limsup_{\e\rightarrow0^+}\frac1\e\(J^{\mu}(t;t,x; u)-J^{\mu}(t;t,x; u^\e\oplus u|_{[t+\e,T]})\)\leq 0$$
	for any $ u^\e\in L^2_\dbF([t,t+\e),U)$, where $J^{\mu}$ is defined in \eqref{cost-3}.
\end{proposition}

\section{Equilibrium}\label{sec-equi}
We are now ready to define an  {\it   equilibrium} for 
the time-inconsistent distribution-dependent problem.
\begin{definition}\label{def-optimalcurve} 
		$\mu^\star\in\sM_\gamma$ is called an {\it   equilibrium} if 
		$$\cT^\gamma_1\circ\cT_2(\mu^\star)=\mu^\star\text{ with } \mu^\star(0)=\gamma.$$
	In addition, 	$u^\star=\cT_2(\mu^\star)$ is called an {\it equilibrium strategy}.
\end{definition}	
Definition \ref{def-optimalcurve} consists of two parts. The first part requires that the distribution curve is the solution under   the corresponding closed-loop strategy. The second part requires   that under the a priori distribution curve, the strategy $u^\star$  is a time-inconsistent  strategy which has been defined in \cite{Yong2012b} and thus verifies the local-optimality  in Proposition \ref{prooptimal}.

The following proposition can be derived directly from Definition \ref{def-optimalcurve}.
\begin{proposition} 
	{\rm (1)} If $\mu^\star$ is an equilibrium,
	$\mu^\star_{[t_1, T]}$ is also an equilibrium with initial $(t_1,\gamma_1)\in\text{\rm Graph}(\mu^\star)$ for any $t_1\in[0,T]$.\ss

	{\rm (2)} For any equilibrium $\mu^\star$ with corresponding strategy $u^\star=\cT_2(\mu^\star)$,
	$$\dbJ(\t;t,\rho;  u^\star)=\int_{\dbR^d} \Theta(\t;t,x)\rho(dx),\q\text{for any }(s,\rho)\in\text{\rm Graph}(\mu),$$
where $\Theta(\tau;s,x)$ is the solution of \eqref{hjb-equi} given $\mu^\star$.

\end{proposition}


\subsection{Existence and Uniqueness of  Equilibrium}\label{sec-equi-existence}
In this subsection,  we  focus on the existence  and uniqueness of the equilibrium for the  time-inconsistent distribution-dependent problem. Obviously, the goal is   to find a fixed point for $\cT^\gamma_1\circ\cT_2$ on $\sM_\gamma$.

  To guarantee the continuity $\cT^\gamma_1\circ\cT_2$ on $\sM$, we require some well-posdeness results of the time-inconsistent HJB equation \eqref{hjb-equi}.
For simplicity, we write 
$$\Ba(t,x,\rho;p)=a(t,x,\rho;\psi(t,x,p))\quad \text{ and }\quad\Bf(\t;t,x,\rho;p)=f(\t;t,x,\rho;\psi(t,x,p)).$$
 Then \eqref{hjb-equi} can be written as
\begin{equation}\label{hjb-equi-1}\left\{\ba{ll}\ad \Theta_t(\t;t,x)+\frac12\text{Tr}[b(t,x,\mu(t))b'(t,x,\mu(t))D^2\Theta(\t;t,x)]\\
\ns\ad\qq+\Ba(t,x,\mu(t);D\Theta(t;t,x))\cdot D\Theta(\t;t,x)+\Bf(\tau;t,x,\mu(t);D\Theta(t;t,x))=0;\\[2mm]
\ns\ad \Theta(\t;t,x)=g(\t;x,\mu(T)).\ea\right.\end{equation}

To prove our main result, we assume the following assumption.
\begin{assumption}\label{A-2} 
		{\rm (a)} For any $\mu\in \sM_\gamma$, there exists  a unique classical solution $\Theta(\t;t,x)$ of \eqref{hjb-equi-1} with  constants $\beta_\Theta^0,\beta_\Theta^1>0$ (independent of $\mu\in\sM_\gamma$) such that
	\bel{theta21bound}\vert D^2\Theta(t;t,x)\vert\leq \beta_\Theta^0\q\text{and}\q |D \Theta(t;t,0)|\leq \beta_\Theta^1
	.\eel
	
	{\rm (b)} Let $\Theta_i(\t;t,x)$ be the solutions  of \eqref{hjb-equi} corresponding to $\mu_i$ for $i=1,2$. There exists a constant $\beta^3_\Theta>0$ such that 
	\begin{equation}\label{nablaW}|D \Theta_1(t;t,x)-D\Theta_2(t;t,x)|\leq \beta^3_\Theta(1+|x|)m(\mu_1,\mu_2).\end{equation}
	%
	
\end{assumption}

Now we are ready to present our main theorem for general Mckean-Vlasov diffusions.
\begin{theorem}\label{mainthm-0} Suppose Assumptions \ref{A-0}, \ref{A-1}, and \ref{A-2} hold. 
	\begin{itemize}
  \item[{\rm (1)}] If $\gamma\in\sP_2(\dbR^d)$ and $\int_{\dbR^d}|x|^{2+\d}\gamma(dx)<\infty$ for some $\d>0$, there exists an   equilibrium.
  \item[{\rm (2)}] If $\gamma\in\sP_2(\dbR^d)$  and $\beta_1\beta_\alpha\beta_\psi\beta^3_\Theta T(1+\int_{\dbR^d}|x|^2\gamma(dx))<1$, where $\beta_\alpha$ is some appropriate constant depending only on $\beta^1_\Theta$ and $\beta_\Theta^2$ to be  defined later, then there exists a unique equilibrium. 

\end{itemize}

\end{theorem}

\begin{proof} 	For any given $\mu\in\sM_\gamma$, write
	$u=\cT_2(\mu)$, i.e., $u(t,x)=\psi(t,x,D\Theta(t;t,x))$.  By \eqref{theta21bound},  we have
	$$|D\Theta(t;t,x)|\leq \beta_\Theta^0|x|+\beta_\Theta^1.$$
	Then   the SDE \eqref{SDE-3} becomes
	$$dX(t)=a(t,X(t),\text{\rm law}(X(t)),\psi(t,X(t),D\Theta(t;t,X(t))))dt+b(t,X(t),\text{ \rm law}(X(t)))dW(t)$$
	with initial $X(0)=\xi$ whose distribution law is $\gamma$.
	By Assumptions \ref{A-0} and  \ref{A-1}, similar to \eqref{ubounsupX}, there exist constants $\beta_\alpha$ and $\beta_{\alpha,\d}$ independent of $\mu$ such that 
	
	$$\BE^u_{t,\xi}\sup_{0\leq t\leq T}|X(t)|^2\leq \beta_\alpha(1+\BE|\xi|^2), 
	\quad  \BE^u_{t,\xi}\sup_{0\leq t\leq T}|X(t)|^{2+\d}\leq \beta_{\alpha,\d}(1+\BE|\xi|^{2+\d}),$$
	and 
	$$\BE^u_{t,\xi}|X(t)-X(s)|^2\leq \beta_\alpha(1+\BE|\xi|^2)|t-s|.$$

	Given any $\mu_1,\mu_2\in\sM_\gamma$,
	write $u_1=\cT_2(\mu_1)$ and $u_2=\cT_2(\mu_2)$. By  \eqref{nablaW} and Assumption \ref{A-1}, it follows that
	$$|u_1(t,x)-u_2(t,x)|\leq \beta_\psi\beta^3_\Theta(1+|x|)m(\mu_1,\mu_2).$$
	As a consequence of \eqref{Xnconvergence}
	\bel{ctctcontra}m(\cT^\gamma_1\circ\cT_2(\mu_1),\cT^\gamma_1\circ\cT_2(\mu_1))\leq\beta_1\beta_\alpha\beta_\psi\beta^3_\Theta Tm(\mu_1,\mu_2)(1+\BE|\xi|^2).\eel

	(1)  If $\gamma$ has a finite $(2+\d)$th moment,
	by Lemma \ref{lawXM},  $\text{\rm LAW}(X)\in\sM^{\d,C,\lambda}_\gamma$ for some $C,\lambda>0$ which are independent of $\mu$. This verifies that 
	$$\cT^\gamma_1\circ\cT_2(\sM^{\d,C,\lambda}_\gamma)\subset\sM^{\d,C,\lambda}_\gamma.$$ Note that 
	$\sM_\gamma^{\d,C,\lambda}$ is a compact and convex set under $m$. In addition, thanks to \eqref{ctctcontra}, $\cT_1^\gamma\circ\cT_2$ is a continuous map. Consequently we can use  Schauder's fixed point theorem to conclude that  there exists at least a $\mu^*$ which is a fixed point of $\cT^\gamma_1\circ\cT_2$. Then $\mu^*$ is the required equilibrium.
	
	(2) If $\beta_1\beta_\alpha\beta_\psi\beta^3_\Theta T(1+\int_{\dbR^d}|x|^2\gamma(dx))<1$, \eqref{ctctcontra} concludes  that $\cT^\gamma_1\circ\cT_2$  is a contraction on $\sM_\gamma$.  Thus there exists a unique  equilibrium in $\sM_\gamma$.
	\end{proof}
\begin{remark}
	\label{remarkproof}
  (1) One can see that our result heavily relies on Assumption \ref{A-2}, which is not a general assumption. We will verify it under some general assumptions later.

	 (1) If $\beta^3_\Theta=0$,  the uniqueness holds directly. A sufficient condition for this case is that $a,f,g$ are independent of the distribution term which reduces to  the time-inconsistent distribution-independent problems investigated in \cite{Yong2012b}. Thus our results generalizes the problem solved there.

\end{remark}

\subsection{Verification of Assumption \ref{A-2}}
In the subsection, we  present a sufficient condition for   Assumption \ref{A-2}. The following  is the  assumption required.
\begin{assumption} \label{B-2}{\rm (1)}  $b$ is independent of $\rho$ and  $b:[0,T]\times\dbR^d$ is continuous with respect to $t$ and has bounded continuous first and second order derivatives with respect to $x$ and there exists a $\lambda>1$ such that
	$$\lambda^{-1}|y|^2\leq y' b(t,x)b' (t,x)y\leq \lambda |y|^2.$$
	
	{\rm (2)}   $a:[0,T]\times\dbR^d\times\sP_2(\dbR^d)\times U\mapsto \dbR^d$, $f:[0,T]\times[0,T]\times\dbR^d\times\sP_2(\dbR^d)\times U\mapsto \dbR$ and $g:[0,T]\times[0,T]\times\dbR^d\times\sP_2(\dbR^d)\mapsto \dbR$ are   continuous and bounded with 
	$$\ba{ll}\ad|a_x(t,x,\rho,u)|+|a_u(t,x,\rho,u)|+|f_x(\t;t,x,\rho,u)|+|g_x(\t;t,x,\rho)|+|g_{xx}(\t;t,x,\rho)|\leq K,\ea$$
	and 
	$$\ba{ll}\ad|a(t,x,\rho_1,u_1)-a(t,x,\rho_2,u_2)|+|f(\t;t,x,\rho_1,u_1)-f(\t;t,x,\rho_2,u_2)|+|g(\t,x,\rho_1)-g(\t,x,\rho_2)|\\
	\ns\ad\q\leq K(d_U(u_1,u_2)+w(\rho_1,\rho_2)).\ea$$
	
\end{assumption}


To proceed, we introduce the  following notations.
For $\b\in(0,1)$, let $C^\b(\dbR^d)$ be the space of function $\f:\dbR^d\to\dbR$ such that $x\mapsto\f(x)$ is
continuous, and
$$\|\f\|_\b:=\|\f\|_0+[\f]_\b<\infty,$$
where
$$\|\f\|_0=\sup_{(x)\in\dbR^d}|\f(x)|,\q[\f]_\b=\sup_{x\ne y}{|\f(x)-\f(y)|\over|x-y|^\b}.$$
Further let $C^{1+\b}(\dbR^d)$ and $C^{2+\b}(\dbR^d)$ be the space of functions $\f:\dbR^d\to\dbR$ such that
$$\|\f\|_{1+\b}=\|\f\|_0+\|\f_x\|_0+[\f_x]_\b<\infty,$$
and
$$\|\f\|_{2+\b}=\|\f\|_0+\|\f_x\|_0+\|\f_{xx}\|_0+[\f_{xx}]_\b<\infty,$$
respectively. Also let $L^\infty(0,T;C^\beta(\dbR^d))$ be the set of all measurable functions $f:[0,T]\times\dbR^d\to\dbR$ such that for fixed $t\in[0,T]$, $f(t,\cd\,)\in C^\b(\dbR^d)$ with
$$\|f(\cd\,,\cd)\|_{L^\infty(0,T;C^\b(\dbR^d ))}=\esssup_{t\in[0,T]}\|f(t,\cd)\|_\b<\infty.$$
Let $C([0,T];C^\b(\dbR^d))$ be the family of continuous functions in $L^\infty(0,T;C^\b(\dbR^d))$. Similarly, we can define
$C([0,T];C^{k+\b}(\dbR^d))\subset L^\infty(0,T;C^{k+\b}(\dbR^d))$.

Then, 
 given any $\mu\in\sM$, we write
$$\left\{\ba{ll}\ad\Ba(t,x,q)=a(t,x,\mu(t);\psi(t,x,q)),\\
\ns\ad \Bb(t,x)=b(t,x),\\
 \ns\ad\Bf(\t;t,x,q)=f(\t;t,x,\mu(t);\psi(t,x,q)),\\
\ns\ad \Bg(\t;x)=g(\t;x,\mu(t).\ea\right.$$
One can see that if $\mu\in\sM$,
$\Ba(t,x,q)$ and $\Bf(\tau;t,x,q)$ are continuous with respect to $t$.

Consider
the following HJB equation,
\begin{equation}\label{HJB-generalcase}\left\{\ba{ll}
\ad  \Theta_t(\t;t,x)+D\Theta(\t;t,x)\cdot\Ba(t,x;D\Theta(t;t,x))+ \Bf(\t;t,x,D \Theta(t;t,x))\\
\ns\ad\q+\frac12\text{Tr}\[\Bb(t,x)\Bb'(t,x)D^2\Theta(\t;t,x)\] =0;\\
\ns\ad \Theta(\t;T,x)=\Bg(\t;x).\ea\right.\end{equation}

\begin{lemma}Under Assumption \ref{B-2}, there exists some constant $K$ which is independent of $\mu\in\sM_\gamma$ such that  the following assertions are true for the classical solution of \eqref{HJB-generalcase}:
	\begin{itemize}
  \item 
  [{\rm (1)}] it holds true that 
	$$|D \Theta(t;t,0)|+| D^2\Theta(t;t,x)|\leq K ;$$
	
	\item[{\rm (2)}] let $\Theta_i(\t;t,x)$ be the solution of of \eqref{HJB-generalcase} corresponding to $\mu_i$, for $i=1,2$. Then
	\begin{equation}\label{nablaW-2}\ba{ll}\ad|D \Theta_1(t;t,x)-D \Theta_2(t;t,x)|\leq Km(\mu_1,\mu_2).\ea\end{equation}
	\end{itemize}
	As a result, Assumption \ref{A-2} holds.
\end{lemma}
\begin{proof} Throughout the proof, $L$ is a generic positive  constant,  independent of $\mu$, whose exact value may change   from line  to line. 
	
Write $\Sigma(s,x):=b(s,x)b'(s,x)$.	Let $\Gamma(t,x;s,y)$ be the fundamental solution of the heat equation
	\begin{equation}\label{heat}\partial_t \Theta(t,x)+\frac12\text{Tr}\(\Sigma(t,x)D^2 \Theta(t,x)\)=0\end{equation}
	with the following representation (see \cite{Fried1964}), 
	$$\G(t,x;s,y)={1\over(4\pi(s-t))^{n\over2}\{\det[\Sigma(s,y)]\}^{1\over2}}
	\exp\left\{{(x-y)^\top \Sigma(s,y)^{-1}(x-y)\over4(s-t)}\right\}.$$
	Tedious  but straightforward calculation yields that
	$$\left\{\ba{ll}\ad|\Gamma(t,x;s,y)|\leq L(s-t)^{-\frac {d}2}\exp\left\{-\frac{\lambda|x-y|^2}{4(s-t)}\right\},\\
	\ns\ad |\Gamma_x(t,x;s,y)|\leq L(s-t)^{-\frac {d+1}2}\exp\left\{-\frac{\lambda|x-y|^2}{8(s-t)}\right\},
	\ea\right.$$
and
\bel{x-y}\G_y(t,x;s,y)=-\G_x(t,x;s,y)+\G(t,x;s,y)\theta(t,x;s,y),\eel
where
$$\left\{\2n\ba{ll}
\ns\ds\theta(t,x;s,y)={(\det[\Sigma(s,y)])_y\over2\det[\Sigma(s,y)])}
+{\lan[\Sigma(s,y)^{-1}]_y(x-y),x-y\ran\over4(s-t)},\\
\ns\ds\lan[\Sigma(s,y)^{-1}]_y(x-y),x-y\ran=\begin{pmatrix}
\lan[\Sigma(s,y)^{-1}]_{y_1}(x-y),x-y\ran\\
\vdots \\
\lan[a(s,y,i)^{-1}]_{y_n}(x-y),x-y\ran\end{pmatrix}.\ea\right.$$
It is easy to check that
\bel{apptheta}|\theta(t,x;s,y)|+|\theta_y(t,x;s,y)|\les K\bigg(1+{|x-y|^2\over s-t}\bigg).\eel
For reader's convenience, note that if $b$ is independent of $x$, then $\theta=0$. This will simplify the proof a great deal.  Here we are dealing with a  general case when $b$  depends on $x$.

(1) To prove that there exists a unique solution to \eqref{HJB-generalcase}, given a $v$,  we consider the following HJB equation 
\begin{equation}\label{HJB-fixpoint}\left\{\ba{ll}
\ad  \Theta_t(\t;t,x)+D\Theta(\t;t,x)\cdot\Ba(t,x;v(t,x))+ \Bf(\t;t,x, v(t,x))\\
\ns\ad\q\q\q\q\q\q\q+\frac12\text{Tr}\[\Bb(t,x)\Bb'(t,x)D^2\Theta(\t;t,x)\] =0;\\
\ns\ad \Theta(\t;T,x)=\Bg(\t;x).\ea\right.\end{equation}
One can easily see that the solution has the following representation:
$$\ba{ll}\Theta(\t;t,x)\ad=\int_{\dbR^d}\Gamma(t,x;T,y)\Bg(\tau;y)dy\\
\ns\ad\q+\int_t^T\int_{\dbR^d}\Gamma(t,x;s,y)D\Theta(\t;s,y)\cdot \Ba(s,y;v(s,y))dyds\\
\ns\ad\q+\int_t^T\int_{\dbR^d}\Gamma(t,x;s,y)\Bf\big(\t;s,y;v(s,y))dyds.\ea$$	
Note that 	$\Theta(t;t,x)$ is continuous with respect to  $t$ because $f$ and $g$ are continuous with respect to $\t$.

(a) First we prove that there exists a constant $K_1$ independent of $v$ such that
$$|D_x\Theta(\t,t,x)|\leq K_1.$$
	Note that \begin{equation}\label{repreWx}\ba{ll}D\Theta(\t;t,x)\ad=\int_{\dbR^d}\Gamma_x(t,x;T,y)\Bg(\t;y)dy\\
\ns\ad\q+\int_t^T\int_{\dbR^d}\Gamma_x(t,x;s,y) D\Theta(\t;s,y)\cdot \Ba(s,y;v(s,y))dyds\\
\ns\ad\q +\int_t^T\int_{\dbR^d}\Gamma_x(t,x;s,y)\Bf\big(\t;s,y;v(s,y)\big)dyds.\ea\end{equation}	
By \eqref{x-y} and \eqref{apptheta}, integrating by parts, we have
$$\ba{ll}\ad\left|\int_{\dbR^d}\Gamma_x(t,x;T,y)\Bg(\tau;y)dy\right|\\
\ns\ad=\left|\int_{\dbR^d}\(-\G_x(t,x;s,y)+\G(t,x;s,y)\theta(t,x;s,y)\)\Bg(\tau;y)dy\right|\\
\ns\ad\leq \left|\int_{\dbR^d}\Gamma(t,x;T,y)\cdot D\Bg(\t;y)dy\right|+\left|\int_{\dbR^d}|\G(t,x;s,y)\theta(t,x;s,y)\Bg(\tau;y)|dy\right|\\
\ns\ad\leq L\(1+\Vert Dg(\t;\cdot)\Vert_{L^\infty}+\Vert g(\t;\cdot)\Vert_{L^\infty}\).\ea$$	
Using Assumption \ref{B-2}, it can be seen that from \eqref{repreWx} that 
$$\ba{ll}|D\Theta(\t;t,x)|\ad\leq L(1+\Vert D\Bg(\tau;\cdot)\Vert_{L^\infty}+\Vert \Bg(\t;\cdot)\Vert_{L^\infty})\\
\ns\ad\q+L\int_t^T(s-t)^{-\frac12}\(1+\|D\Theta(\t;s,\cdot)\|_{L^\infty}\)ds\\
\ea$$	
Using Grownwall's inequality, we have 
\begin{equation}\label{WDWbound}\|D\Theta(\t;t,\cdot)\|_{L^\infty}\leq K\(1+\sup_{0\leq\t\leq T}\(\Vert D\Bg(\tau;\cdot)\Vert_{L^\infty}+\Vert \Bg(\tau;\cdot)\Vert_{L^\infty}\)\):=K_1.\end{equation}	
where $K$ is a constant independent of $\mu$ and $v$.\ss

	(b) Let $\Theta_i$ be the solution of  \eqref{HJB-fixpoint} under $v_i$. We need to  prove that for some $K_2>0$
\bel{vVcontraction}\sup_{T-\delta\leq t\leq T}|D\Theta_1(\t;t,\cdot)-D\Theta_2(\t;t,\cdot)|_{L^\infty}\leq K_2\sqrt{\d}\(1+\Vert v(s,\cdot)-v(s,\cdot)\Vert_{L^\infty}\).\eel

Note that
$$\begin{aligned}|&D\Theta_1(\t;t,x)-D\Theta_2(\t;t,x)|\\
& \leq \int_t^T\bigg|\int_{\dbR^d}\Gamma_x(t,x;s,y)D\Theta_1(\t;s,y)\cdot \Ba(s,y;v_1(s,y))dy \\ & \qquad \qquad\qquad -\int_{\dbR^d}\Gamma_x(t,x;s,y)D\Theta_2(\t;s,y)\cdot \Ba(s,y;v_2(s,y))dy\bigg| ds\\
&\q+\int_t^T \left|\int_{\dbR^d}\Gamma_x(t,x;s,y) \Bf(\t,s,y;v_1(s,y))dy-\int_{\dbR^d}\Gamma_x(t,x;s,y) \Bf(\t,s,y;v_2(s,y))dy\right|ds\\
&\leq K\int_t^T(s-t)^{-\frac12}\(1+\Vert D\Theta_1(\t;s,\cdot)-D\Theta_2(\t;s,\cdot)\Vert_{L^\infty}+\Vert v_1(s,\cdot)-v_2(s,\cdot)\Vert_{L^\infty}\)ds.\end{aligned}$$
This concludes \eqref{vVcontraction}.\ss

From \eqref{repreWx}, we know $\Theta(\t,t,\cdot)\in C^1(\dbR^d)$. Define a map $\Psi:C([T-\d,T],C^1(\dbR^d))\mapsto C([T-\d,T],C^1(\dbR^d))$ such that $\Phi(\Theta)$ is the solution of  \eqref{HJB-fixpoint} using $D\Theta$.
By \eqref{WDWbound} and \eqref{vVcontraction}, we know that $\Psi$  is a contraction if $\d$ is small. Thus there exists a unique solution $\Theta$ for \eqref{HJB-generalcase} on $[T-\d, T]$. Since the constant $K$ is independent of $\d$, for time interval $[0,T]$. One can divide the time horizon $[0,T]$ into several small intervals, and then 
prove that the solution exists on whole time interval $[0,T]$ recursively.

(2) By \eqref{WDWbound}, it is easy to see that 
$$\Vert D\Theta(\t,s,\cdot)\Vert_{L^\infty}\leq K.$$ 
Now we verify that
	\bel{boundfortheta}  \Vert D^2\Theta(\t,s,\cdot)\Vert_{L^\infty}\leq K.\eel

Use the fundamental solution method again, writing $v=D\Theta$,
\bel{theta2der} \ba{ll}\partial_{x_i,x_j}\Theta_{n+1}(\t;t,x) \ad=\int_{\dbR^d}\Gamma_{x_i,x_j}(t,x;T,y)\Bg(\t;y)dy\\
\ns\ad\q+\int_t^T\int_{\dbR^d}\Gamma_{x_i,x_j}(t,x;s,y)D\Theta_{n+1}(\t,s,y)\cdot \Ba(s,y,\rho(s);v(s,y))dyds\\
\ns\ad\q+\int_t^T\int_{\dbR^d}\Gamma_{x_i,x_j}(t,x;s,y)\Bf\big(\t;s,y;v(s,y)\big)dyds.\ea\eel
Note that 
$$\ba{ll}\ad \Gamma_{x_i,x_j}(t,x;s,y) \\ \ad\ \ =\(-\Gamma_{y_i}(t,x;s,y)+\Gamma(t,x;s,y)\theta_i(t,x,s,y)\)_{x_j}\\
\ns\ad\ \ =-\Gamma_{x_j,y_i}(t,x;s,y)+\Gamma_{x_j}(t,x;s,y)\theta_i(t,x,s,y)+\Gamma(t,x;s,y)\cdot\frac{[\Sigma(s,y)]^{-1}]_{y_j}(x_j-y_j)}{2(s-t)},\ea$$
it follows that
\begin{equation}\label{contraction-1}\ba{ll}\ad\left|\int_{\dbR^d}\Gamma_{x_i,x_j}(t,x;T,y)\Bg(\t;y)dy\right|\\
\ns\ad\q\leq \bigg|\int_{\dbR^d}\(-\Gamma_{x_j,y_i}(t,x;s,y)+\Gamma_{x_j}(t,x;s,y)\theta_i(t,x,s,y) \\ \ad \qquad \qquad+\Gamma(t,x;s,y)\cdot\frac{[\Sigma(s,y)]^{-1}]_{y_j}(x_j-y_j)}{2(s-t)}\)\Bg(\t,y)dy\bigg|\\
\ns\ad\q\leq \left|\int_{\dbR^d}\Gamma_{x_j}(t,x;s,y)\Bg_{y_j}(\t,y)dy\right|\\
\ns\ad\qq+\left|\int_{\dbR^d}\(-\G_{y_j}(t,x;s,y)+\G(t,x;s,y)\theta_j(t,x;s,y)\)\theta_i(t,x,s,y)\Bg_{y_j}(\t,y)dy\right|\\
\ns\ad\qq+\left|\int_{\dbR^d}\Gamma(t,x;s,y)\cdot\frac{[\Sigma(s,y)]^{-1}]_{y_j}(x_j-y_j)}{2(s-t)}\Bg(\t,y)dydy\right|\\
\ns\ad\q \leq K(1+\Vert D^2g(\t,\cdot)\Vert_{L^\infty}+\Vert Dg(\t,\cdot)\Vert_{L^\infty}+\Vert g(\t,\cdot)\Vert_{L^\infty}).
\ea\end{equation}
Note that \bel{contraction-2}\ba{ll}\ad\left|\int_{\dbR^d}\Gamma_{x_i,x_j}(t,x;s,y)D\Theta(\t,s,y)\cdot \Ba(s,y,;v(s,y))dy\right|\\
\ns\ad\leq \left|\int_{\dbR^d}\Gamma_{x_j}(t,x;s,y) \(D\Theta(\t,s,y)\Ba(s,y,;v(s,y))\)_{y_i}dy\right|\\
\ns\ad\q+\left|\int_{\dbR^d}\Gamma_{x_j}(t,x;s,y) D\Theta(\t,s,y)\Ba(s,y;v(s,y))dy\right|\\
\ns\ad\q+\left|\int_{\dbR^d}\Gamma(t,x;s,y)\cdot\frac{[\Sigma(s,y)]^{-1}]_{y_j}(x_j-y_j)}{2(s-t)}D\Theta(\t,s,y)\Ba(s,y;v(s,y))dy\right|\\
\ns\ad\leq L(t-s)^{-\frac12}(1+\Vert D^2\Theta(\t,s,\cdot)\Vert_{L^\infty}+\Vert D\Theta(\t,s,\cdot)\Vert_{L^\infty})+\Vert Dv(\t,s,\cdot)\Vert_{L^\infty}).\ea\eel
Smilarly we have 
\bel{contraction-3}\ba{ll}\ad\left|\int_{\dbR^d}\Gamma_{x_i,x_j}(t,x;s,y) \Bf(\t,s,y;v(s,y))dy\right|
\leq L(t-s)^{-\frac12}(1+\Vert Dv(\t,s,\cdot)\Vert_{L^\infty}).\ea\eel

Recall $v=D\Theta$, plugging \eqref{contraction-1}--\eqref{contraction-3} into \eqref{theta2der}, and noting \eqref{WDWbound}, it follows that
$$ 
\ds|\partial_{x_i,x_j}\Theta(\t;t,x)|\leq L\(1
+\int_{t}^T(s-t)^\frac12\(1+\Vert D^2\Theta(\t;s,\cdot)\Vert_{L^\infty}\)ds\).
$$ 
Gronwall's inequality then implies that \eqref{boundfortheta}  holds.

	(3)	 Now let's verify \eqref{nablaW-2}. Let $\Theta_i$ be the solutions of \eqref{HJB-generalcase} using $\mu_i$ for $i=1,2$. By  \eqref{repreWx} and \eqref{WDWbound},
	$$\ba{ll}\ad|D\Theta_1(\t;t,x)-D\Theta_2(\t;t,x)|\\
	\ns\ad\leq L\int_{\dbR^d}|\Gamma(t,x;s,y)-\Gamma(t,x;s,y)|\cdot|D\Bg(\tau,y)|dy\\
	\ns\ad\q+ L\int_t^T\int_{\dbR^d}|\Gamma_{x}(t,x;s,y)||D \Theta_1(\t;s,y)-D\Theta_2(\t;s,y)|dyds\\
	\ns\ad\q+L\int_t^T\int_{\dbR^d}|\Gamma_{x}(t,x;s,y)|w(\mu_1(s),\mu_2(s))dyds
	\\
	\ns\ad\leq L\int_t^T(s-t)^{-\frac12}\sup_{0\leq \t\leq T}\|D \Theta_1(\t;s,\cdot)-D \Theta_2(\t;s,\cdot)\|_{L^\infty}ds+K(T-t)^\frac12m(\mu_1,\mu_2).\ea$$
	By Grownwall's inequality, we have 
	$$|D\Theta_1(\t;t,x)-D\Theta_2(\t;t,x)|\leq Km(\mu_1,\mu_2).$$
	The proof is complete.
\end{proof}

\section{Semi-linear Distribution-dependent Case}\label{sec-semi-linear}
In this section, we  deal with a special case of semi-linear distribution-dependent diffusions with a quadratic cost.  In such a case,    the solution to the HJB equation can be presented as a Riccati equation  which simplifies the verification process a lot. By such an explicit representation,  the system is not required to be non-degenerate and therefore  $W(\cdot)$ is assumed to be 1-dimensional Brownian motion for convenience. Moreover, to simplify the form of the Riccati equation,  we  assume the diffusion coefficient depends on the distribution term  and the  time variable $t$ only, not on the state variable $x$. 

 In this section,   $\dbS_d$ is the  set of   $d\times d$ symmetric matrices equipped with following metric
$$|S_1-S_2|_\dbS=\sup_{|x|=1}|\langle x,(S_1-S_2)x\rangle|, \quad S_{1},S_{2}\in \dbS_{d}.$$
Next $C([0,T],\dbS)$ denotes the set of  $\dbS$-valued continuous curves on $[0,T]$ equipped with norm
$$\Vert S\Vert_{C([0,T],\dbS)}:=\sup_{0\leq t\leq T}|S(t)|_{\dbS},$$
where 
$$ S=\{S(t)\in\dbS:0\leq t\leq T\}.$$

Consider the following $d$-dimensional distribution-dependent controlled SDE
\begin{equation}\label{semilinear}\ba{ll}\ds dX(t)\ad=[A(t)X(t)+B(t)u(t)+a(t,\rho(t))]dt+b(t,\rho(t))dW(t),\ea\end{equation}
where $A(\cdot) \in\dbR^{d\times d}$, $B(\cdot)\in\dbR^{d\times l}$ and $a(\cdot),b(\cdot):[0,T]\times\sP_2(\dbR^2)\mapsto\dbR^{d}$.  The control space is $U=\dbR^l$.

 The admissible strategy is defined as in \eqref{admissiblecontrol}. Let
 $$f(\t;t,x,\rho;u)=\langle x,Q(\t;t)x\rangle+ \langle u,R(\t;t)u\rangle,\q g(\t;x,\rho)=\langle x, G(\t)x\rangle.$$
The time-inconsistent cost functional is defined as
\begin{equation}\label{cost-2}\ba{ll}\ad\dbJ(\t;t_0,\xi; u)\\
\ns\ad:=\BE_{t_0,\xi}\bigg[\int_{t_0}^T f(\t;t,X(t);u(t,X(t))+F(\t;t,\text{law}(X(t))) dt+ g(\t;X(T))+H(\t;\text{law}(X(T)) \bigg],\ea\end{equation}
where $F:[0,T]\times[0,T]\times\sP_2(\dbR^d)\mapsto \dbR^+$ and $H:[0,T]\times\sP_2(\dbR^d)\mapsto \dbR^+$.
 The value function is
\begin{equation}\label{value-2}\dbV(t_0,\xi; u)=\dbJ(t_0;t_0,\xi;u).
\end{equation}

We need the following assumption in this section.
\begin{assumption}
\label{A-3}
\begin{itemize}
	\item[{\rm (1)}]
$$\left\{\ba{ll}
\ad|a(t,\rho)|^2+|b(t,\rho)|^2+F(\t;t,\rho)+H(\t;\rho)\leq K\(1+\int_{\dbR^d}|y|^2\rho(dy)\)\\
\ns\ad|a(t,\rho_1)-a(t,\rho_2)|+|b(t,\rho_1)-b(t,\rho_2)|\leq Kw(\rho_1,\rho_2)\\
\ns\ad|F(\t;t,\rho_1)-F(\t;t,\rho_2)|^2+|H(\t;\rho_1)-H(\t; \rho_2)|^2\\
\ns\ad\qq\leq K\(1+\int_{\dbR^d}|y|^2\rho_1(dy)+\int_{\dbR^d}|y|^2\rho_2(dy)\)w^2(\rho_1,\rho_2).\ea\right.$$

	\item[{\rm (2)}]	$ A(\cdot),B(\cdot)$ are bounded continuous deterministic functions on $[0,T]$.

	\item[{\rm (3)}] $ Q(\t;\cdot),R(\t;\cdot)$ are  uniformly bounded continuous $\dbS$-valued deterministic processes  on $[0,T]$ with
$Q(\t;t),R(t;t)\geq\e_0^{-1}I_l
$ for some $\e_0>0$.
	\item[{\rm (4)}] $G(\t)\geq0$.

\end{itemize}
\end{assumption}

\subsection{Linear Distribution-independent Diffusion}

Given  ${\Ba}(\cdot),\Bb(\cdot),\BF(\t,\cdot)\in C([0,T],\dbR^{d})$ and $\BH(\t)\in\dbR$, consider the following $d$-dimensional  controlled SDE with $W(\cdot)$ being a 1-dimensional standard Brownian motion,
\bel{semisdemu} dX(t)=[A(t)X(t)+B(t)u(t)+\Ba(t)]dt+\Bb(t)dW(t)\eel
with the  cost functional defined as \begin{equation}\label{cost-2}\ba{ll}\ns\ad J(\t;t,x;u)\\
\ns\ad:=\BE_{t,x}\bigg[\int_t^T [\langle X(s),Q(\t;t)X(s)\rangle+ \langle u(s),R(\t;s)u(s)\rangle+ \BF(\t,s)]ds+\langle G(\t)X(T),X(T)\rangle+\BH(\t)\bigg]\ea\end{equation}
and the value function is
$$V(t,x;u)=J(t;t,x;u).$$

Since it is linear-quadratic case now, solving \eqref{hjb-equi} using the appropriate coefficients, one can conclude that the time-inconsistent equilibrium is
\begin{equation}\label{equilq}u(t,x)=-R(t;t)^{-1}B'(t) [P(t;t)x+ p(t;t)],\end{equation}
where 
 $P(\t;\cdot)\in\dbR^{d\times d},q(\t;\cdot)\in\dbR^d,\eta(\t;\cdot)\in\dbR$ satisfy (we omit the dependence on $u$ in $J$ and $V$ now)
$$J(\t;t,x)=\langle P(\t;t)x,x\rangle+2\langle p(\t;t),x\rangle+\eta(\t;t )\q\text{and}\q V(t,x)=J(t;t,x),$$
and $(P,p,\eta)$ satisfies the following Riccati equations (if there exists a solution)
\begin{equation}\label{riccati-1}\left\{\ba{ll}\ad\dot{P}(\t;t)+P(\t;t) A(t)+ A'(t) P(\t;t)+Q(\t;t)-2P(\t;t)B(t)R(t;t)^{-1}B'(t)P(t;t)\\
\ns\ad\qq +P(t;t)B(t)R(t;t)^{-1}R(\t;t)R(t;t)^{-1}B'(t)P(t;t)=0,\\
\ns\ad P(\t;T)=G(\t);\ea\right.\end{equation}
\begin{equation}\label{riccati-2}\left\{\ba{ll}\ns\ad\dot p(\t;t)+P(\t;t)[\Ba(t)-BR(t;t)^{-1}B'(t)p(t;t)]\\
\ns\ad\q+P(t;t)B(t)R^{-1}(t;t)R(\t;t)R^{-1}(t;t)B'(t)P(t;t)=0,\\
\ns\ad p(\t;T)=0;\ea\right.\end{equation}
and
\begin{equation}\label{riccati-3}\left\{\ba{ll}\ns\ad \dot\eta(\t;t)+\lan P(\t;t)\Bb(t),\Bb(t)\ran+p'(\t;t)B(t)R(t;t)^{-1}R(\t;t)R(t;t)^{-1}B'(t)p(\t;t)+\BF(\t;t)=0,\\
\ns\ad \eta(\t;T)=\BH(\t).\ea\right.\end{equation}

Here we note that $u$ is depending on $\mu$ which doesn't meet Assumption \ref{A-1}. While due to the linear structure, it is still possible for us to deal with such a special case.

Based on the representations, we have the following proposition.

\begin{proposition}\label{verifyassumption}

\begin{itemize}
	\item[\rm(1)] $P(\t;t)$ is independent of the choices of $\Ba,\Bb,\BF,\BH$. So is $D^2V(t,x)$.

	\item[\rm(2)]  There exists a positive constant $K_P$ (depending on $P$) such that
$$\vert D V(t,0)\vert=2\vert p(t;t)\vert\leq K_PT(1+\Vert \Ba(\cdot)\Vert_{L^\infty}).$$

	\item[\rm(3)] Let $V_i(t,x)$  be the solutions corresponding to $(\Ba_i(\cdot),\Bb_i(\cdot)), i=1,2$. Then
$$\vert DV_1(t,x)-D V_2(t,x)\vert=2\vert p_1(t;t)-p_2(t;t)\vert\leq K_PT\Vert \Ba_1(\cdot)-\Ba_2(\cdot)\Vert_{L^\infty}.$$

\end{itemize}
\end{proposition}

\begin{theorem}
	Suppose that \eqref{riccati-1} admits a unique solution. If $a$ and $b$ are bounded and $\gamma$ has finite $(2+\d)$th moment for some $\d>0$, then there exists an equilibrium.
\end{theorem}

\begin{proof}
Let $X$ be the solution of \eqref{semisdemu} with 	$\Ba(t)= a(t,\mu(t)), \Bb(t)= b(t,\mu(t))$
Since $a$ and $b$ are bounded, It\^{o}'s formula concludes that 
\bel{uboundXXX}\BE\sup_{0\leq t\leq T}|X(t)|^{2+\d}\leq K\eel	
	for  some $K>0$ independent of $\mu$.
	
Let
\begin{equation}\label{BbBd}\Ba_i(t)= a(t,\mu_i(t)),~\Bb_i(t)= b(t,\mu_i(t)),~\BF_i(t)= F(t,\mu_i(t))\text{ and }\BH_i(\t)= H(\t,\mu_i(T))\end{equation}
Thus it follows that
$$V_i(\t;t,x)=\langle P_i(\t;t)x,x\rangle+2\langle p_i(\t;t),x\rangle+\eta_i(\t;t )$$
where $(P_i,p_i,\eta_i)$ is the solution of \eqref{riccati-1}, \eqref{riccati-2} and \eqref{riccati-3} with $(\Ba_i,\Bb_i,\BF_i,\BH_i)$.
By the definitions of $\Ba$ and $\Bb$ in \eqref{BbBd}, Assumption \ref{A-3} and Proposition \ref{verifyassumption} yields the following  estimates
	\bel{plip}\left\{\ba{ll}
	\ad\vert p(t;t)\vert^2\leq K_PT\(1+\Vert \Ba(\cdot)\Vert\)\leq K\\ [2mm]
	\ns\ad\vert p_1(t;t)-p_2(t;t)\vert\leq KTm(\mu_1,\mu_2).\ea\right.\eel
By virtue of the proof for Theorem \ref{mainthm-0}, there exists  an equilibrium.	
	\end{proof}

\begin{remark}
	{\rm (1) In \cite{Yong2012b}, the author presented a sufficient condition in the examples for the existence of a  solution to \eqref{riccati-1}. 
		
	(2)	 The assumption that $a$ and $b$ are bounded is not very general because  it didn't even cover the case $$a(t,\rho)=\int_{\dbR^d}x\rho(dx).$$  The reason for such assumption is to guarantee that the first inequality in \eqref{plip} holds such that \eqref{uboundXXX} is true.

(3) It is not hard to see  if $T$ is small, \eqref{riccati-1} has a unique solution and \eqref{plip} holds. Thus our results are always true for small time horizon.}
\end{remark}

\subsection{Strong Dissipative Case}\label{sec-long-time}

In this section, we will raise a strong dissipativity condition such that 

(1) $a$ and $b$ are not  necessarily bounded,

(2) \eqref{riccati-1} admits a unique solution,

(3) \eqref{plip} and \eqref{uboundXXX} hold,

(4) the equilibrium is unique.

\noindent The following is the assumption we use.
\begin{assumption}\label{B-0}
Let $\lambda_{\max}(t)$ be the largest real part of the eigenvalues of $A(t)$. Assume that for some $L_0>0$,
$$\sup_{0\leq t\leq T}\lambda_{\max}(t)\leq -L_0.$$ 

\end{assumption}

\begin{lemma} Under Assumption \ref{B-0}, if $L_0>0$ is large,
\eqref{riccati-1} exists a unique solution. Moreover, the solution $P(\t;t)$ is uniformly bounded by a constant which is independent of $L_0$. As a result, the constant $K_P$ in Proposition \ref{verifyassumption} is independent of $L_0$.
\end{lemma}
\begin{proof} 
We adopt the fixed-point theory here.
Given $v_i(t)= P_i(t;t)$ (resp. $\bar v_i(t)=\bar P_i(t;t)$), let $ P_{i+1}(\t;t)$ (resp. $\bar P_{i+1}(\t;t)$) be the solution of 	
\begin{equation}\label{ricin}\left\{\ba{ll}\ad\dot{ P }(\t;t)+ P (\t;t) A+A' P (\t;t)+Q(\t;t)- 2 P (\t; t)BR(t;t)^{-1}B'v_i(t)\\
\ns\ad\qq +v_i(t)BR(t;t)^{-1}R(\t;t)R(t;t)^{-1}B'v_i(t)=0;\\
\ns\ad  P (\t;T)=G(\t).\ea\right.\end{equation} 
Note that 
$$\left\{\ba{ll}\ad 0\leq  x'\big(v_i(t)BR(t;t)- P (\t;t)BR(\t;t)\big)R(\t;t)^{-1}\big(R(t;t)B'v_i(t)-R(\t;t)B' P (\t;t)\big)x\\
\ns\ad\qq\leq L(|P(\t;t)|^2_\dbS+|v_i(t;t)|^2_\dbS)|x|^2,\\ [2mm]
\ns\ad x' P (\t;t)BR(\t;t)^{-1}B' P (\t;t)x\geq 0,\ea\right.$$
and thus
\begin{align*}0& =\dot{ P }_{i+1}(\t;t)+ P_{i+1} (\t;t) A+A' P_{i+1} (\t;t)+Q(\t;t)- 2 P _{i+1}(\t; t)BR(t;t)^{-1}B'v_i(t) \\
\ns\ad\q+v_i(t)BR(t;t)^{-1}R(\t;t)R(t;t)^{-1}B'v_i(t)\\
& =\dot{ P_{i+1} }(\t;t)+ P_{i+1} (\t;t) A+A' P_{i+1} (\t;t)+Q(\t;t)- P_{i+1} (\t;t)BR(\t;t)^{-1}B' P_{i+1} (\t;t)\\
\ns\ad\q+\big(v_iBR(t;t)- P_{i+1} (\t;t)BR(\t;t)\big)R(\t;t)^{-1}\big(R(t;t)B'v_i-R(\t;t)B' P_{i+1} (\t;t)\big)\\
& \leq \dot{ P_{i+1} }(\t;t)+ P_{i+1} (\t;t) A+A' P (\t;t)+Q(\t;t)+L(|P_{i+1}(\t;t)|^2_\dbS+|v_i(t)|^2_\dbS)I_d.\end{align*}
Let  $\cA(t)$ be the solution of
$$\frac{d}{dt}\cA(t)=A(t) \q\text{and}\q \cA(T)=0.$$
Note that 
$$\ba{ll}\ds\frac d{dt}\(e^{-\cA'(t)}P_{i+1}(\t;t)e^{-\cA(t)}\)\ad=e^{-\cA'(t)}\(\dot P_{i+1}(\t;t)+A(t)'P_{i+1}(\t;t)+P_{i+1}(\t;t)A(t)\)e^{-\cA(t)}\\
\ns\ad\geq -e^{-\cA'(t)}\(Q(\t;t)+L(|P_{i+1}(\t;t)|^2_\dbS+|v(t)|^2_\dbS)I_d)e^{-\cA(t)}.\ea$$
Therefore one can conclude that
$$\ba{ll}P_{i+1}(\t;t)\ad\leq e^{\cA'(t)}G(\t)e^{\cA(t)}+e^{\cA'(t)}\int_t^Te^{-\cA'(s)}\(Q(\t;s)+L(|P_{i+1}(\t;t)|^2_\dbS+|v_i(t)|^2_\dbS)I_d\)e^{-\cA(s)}ds\\
\ns\ad\leq e^{\cA'(t)}G(\t)e^{\cA(t)}+\int_t^Te^{\cA'(t)-\cA'(s)}\(Q(\t;s)+L(|P_{i+1}(\t;t)|^2_\dbS+|v_i(t)|^2_\dbS)I_d\)e^{\cA(t)-\cA(s)}ds\\
\ns\ad\leq K_0\(K_1+\big(\Vert P_{i+1}(\t;\cdot)\Vert^2_{C([t,T];\dbS)}+\Vert v_{i}(\cdot)\Vert^2_{C([t,T];\dbS)}\big)\int_t^Te^{-L_0(s-t)}ds\)I_d\\
\ns\ad\leq K_0\(K_1+\frac{1}{L_0}(\Vert P_{i+1}(\t;\cdot)\Vert^2_{C([t,T];\dbS)}+\Vert v_{i}(\cdot)\Vert^2_{C([t,T];\dbS)})\)I_d\ea$$
By Assumption \ref{B-0}, there exist  uniform constants $K_0,K_1$ (independent of $L_0$),
\begin{equation}\label{vbound-0}\Vert  P_{i+1}(\t;\cdot)\Vert_{C([0,T];\dbS)}\leq K_0\(K_1+\frac{1}{L_0}\Vert  P_i(\t;\cdot)\Vert^2_{C([0,T];\dbS)}\).\end{equation}
If $L_0$ is large such that $L_0\geq 4K_0^2K_1$ and $\Vert v_0(\cdot)\Vert_{C([0,T];\dbS)}$ is small, we can conclude that 
\begin{equation}\label{vbound}\ba{ll}\Vert  P_{i}(\t;\cdot)\Vert_{C([t_0,T];\dbS)}\ad\leq \frac12\(\frac{L_0}{K_0}-\sqrt{\(\frac {L_0}{K_0}\)^2-4L_0K_1}\)\\[2mm]
\ns\ad\leq \frac12\frac{4L_0K_1}{\frac{L_0}{K_0}+\sqrt{\(\frac {L_0}{K_0}\)^2-4L_0K_1}}\\
\ns\ad\leq 2K_1K_0.\ea\end{equation}
Note that the constant on the right-hand side is independent of $L_0$.

By \eqref{ricin}, simple calculation yields that 
$$\ba{ll}\ad\frac{d}{dt}( P_{i+1} (\t;t)-\bar P_{i+1} (\t;t))+( P_{i+1} (\t;t)-\bar P_{i+1} (\t;t))A+A'( P_{i+1} (\t;t)-\bar P_{i+1} (\t;t))\\
\ns\ad\q -2 P_{i+1} (\t;t)BR(\t;t)^{-1}B' v_i (t)+2\bar P_{i+1} (\t;t)BR(\t;t)^{-1}B'\bar v_i (t)\\
\ns\ad\q+ v_i(t)BR(t;t)^{-1}R(\t;t)R(t;t)^{-1}B' P (t;t)-\bar v_i (t)BR(t;t)^{-1}R(\t;t)R(t;t)^{-1}B'\bar v_i(t)=0.\ea$$
Using the uniform bound in \eqref{vbound},  we can conclude that
$$\Vert  P_{i+1} (\t;\cdot)- \bar  P_{i+1} (\t;\cdot)\Vert_{C([T-\d,T];\dbS)}\leq K\d \Vert  v_i (t)-\bar  v_i (t)\Vert_{C([T-\d,T];\dbS)}.$$
Thus if $\d$ is small, there exists a unique solution of \eqref{riccati-1} on $[T-\d,T]$. For a general time interval $[t_0,T]$, we can divide the interval into $[T-\d,T],[T-\d,T-2\d],\cdots$. One can conclude that there exists a unique solution on $[0,T]$. By the form of the constant in \eqref{vbound}, we know the bound of $P(\t;t)$ is independent of $L_0$.
	\end{proof}

Now we are ready to present our result for dissipative semi-linear diffusions.

\begin{theorem}\label{mainlemma-2} Under Assumption \ref{A-3} and \ref{B-0}, if $L_0$ is large,
	there exists a unique equilibrium $\mu$.
\end{theorem}

\begin{proof} Throughout the proof,  $L$ is a generic  constant varying from place to place but independent of $L_0$.
	Let $\mu_1(t)=\gamma$ and $u_1(t,x)=-R(t;t)^{-1}B'(t) [P(t;t)x+ p_1(t;t)]$, where
	$p_1(\cdot;\cdot)$ is the solution of \eqref{riccati-2} using $\Ba(t)=a(t,\mu_1(t)).$ Then let $\mu_2$ be distribution curve of the solution $X_2$ of \eqref{semilinear} using strategy $u_1$. Recursively repeating such processes, we can get a sequence of $\{\mu_n:\mu_n(0)=\gamma\}$.
	 Note that in the 2-step recursion method,
	$$u_i(t,x)=-R(t;t)^{-1}B(t)' [P(t;t)x+ p_{i-1}(t;t)].$$
	By Assumption  \ref{A-3}, using Ito's formula, we have
	$$\ba{ll}\ds\frac d{dt}\BE|X_i(t)|^2\ad=-2(L_0-L)\BE|X_i(t)|^2+L(1+| p_{i-1}(t;t)|^2)\\
	\ns\ad\leq -2(L_0-L)\BE|X_i(t)|^2+L(1+\sup_{0\leq t\leq T}\BE|X_{i-1}(t)|^2)\ea$$
	Grownwall's inequality implies that
	$$\sup_{0\leq t\leq T}\BE|X_i(t)|^2\leq \frac{L}{L_0-L}\sup_{0\leq t\leq T}\BE|X_{i-1}(t)|^2+L(1+\BE|\xi|^2)$$
	If $L_0>0$ is large, we can  see that 
	\begin{equation}\label{Xbound-2} \sup_{0\leq t\leq T}\BE|X_i(t)|^2\leq L(1+\BE|\xi|^2).\end{equation}

	Using It\^o's formula, we have
	\bea\ba{ll}\ad \frac{d}{dt}\BE|X_1(t)-X_2(t)|^2\\ [2mm]
	\ns\ad\leq (-2L_0+L)\BE|X_1(t)-X_2(t)|^2+L\vert p_1(t;t)-p_2(t;t)\vert^2\\ [2mm]
	\ns\ad \leq (-2L_0+L)\BE|X_1(t)-X_2(t)|^2+L\sup_{0\leq s\leq T}\BE|X_0(s)-X_1(s)|^2.\ea\eea
	where we use the following 
	$$|b(t,\mu_1(t))-b(t,\mu_2(t))|^2\leq Lw^2(\rho_1(t),\rho_2(t))\leq L\BE|X_1(t)-X_2(t)|^2.$$
	 Grownwall's inequality implies
	$$\ba{ll}\ds
	\sup_{0\leq t\leq T}\BE|X_1(t)-X_2(t)|^2\leq \frac{L}{L_0-L}\sup_{0\leq t\leq T}\BE|X_0(t)-X_1(t)|^2.\ea$$
	If $L_0$ is large, by Lemma \ref{lemmaXrho},  $\mu_n$ is a Cauchy sequence in $\sM_\gamma$ and its limit is the unique equilibrium.
	\end{proof}

\section{Mean-field Game}\label{sec:mfg}
In this section, we compare our results with a mean-field game  for  infinite-many symmetric players. We use the same notations from the previous sections.
 
For $i=1,\cdots, N$,  the dynamic for $i$th player is
\bel{SDE-game}dX_i(t)=a(t,X_i(t),\mu^{-i}_N(t), u_i(t))dt+b(t,X_i(t),\mu^{-i}_N(t))dW_i(t)\eel
where $\mu_{N}^{-i}(t,dx)=\frac{1}{N-1}\sum_{j\neq i}\d_{X_j(t)}(dx)$. The $i$th player makes his decision based the following time-inconsistent cost functional
$$\dbV^i(t,x,\mu_N^{-i};u_i)=\dbJ(t;t,x,\mu_N^{-i};u_i),$$
where
$$\dbJ(\t;t,x,\mu;u):=\dbE_{t,x}\bigg[\int_t^Tf(\t;s,X(s),\mu(s);u(s))ds+g(\t;X(T),\mu(T))\bigg]$$
and $X(s)$ is the solution of \eqref{SDE-game} with initial $x$.

Since the cost functional is time-inconsistent, the players shall look for a local optimal strategy instead of a global one. Since all the players are symmetric, we would suppose that every player should obey  the same strategy. Letting $N\rightarrow\infty$, by the law of large numbers, we can define the {\it equilibrium}  and the corresponding (closed-loop) {\it equilibrium strategy} as following.

\begin{definition}
$\mu^\star\in\sM_\gamma$ is called an equilibrium and 	$u^\star:[0,T]\times\dbR^d\mapsto U$ is called a (closed-loop) equilibrium strategy  if
\begin{itemize}
  \item[{\rm(1)}] $\mu^\star$ is the distribution curve of the following SDE,
$$dY(t)=a(t,Y(t),\mu^\star(t), u^\star(t,Y(t))))dt+b(t,Y(t),\mu^\star(t))dW(t), \q\text{\rm law}(Y(0))=\gamma.$$
 \item[{\rm(2)}] the following local optimality holds,	$$\limsup_{\e\rightarrow0^+}\frac{\dbJ(t;t,x,\mu^\star;u^\star)-\dbJ(t;t,x,\mu^\star;u^\e\oplus u^\star|_{[t+\e,T])}}{\e}\leq0$$ for any $(t,x)\in[0,T]\times\dbR^d$ and $u^\e\in L^2_{\dbF}([t,t+\e),U).$ 
  \end{itemize}
\end{definition}

By Proposition \ref{prooptimal}, the equilibrium defined in Definition \ref{def-optimalcurve} is same as the equilibrium for such a  mean-field game with infinite-many symmetric players. Since $f$ and $g$ here can be fully determined by the mean-field game, such equivalence also illustrates the confusion (3) mentioned in Remark \ref{remconfu}. 

\section{Concluding Remarks}\label{sec:conrem}
In this paper, we proved the existence and uniqueness of an equilibrium for general time-inconsistent McKean-Vlasov dynamics  and a special semi-linear case under some appropriate assumptions. The results generalized the results in \cite{Yong2012b} for  McKean-Vlasov dynamics. Moreover, the equilibrium  coincides with the   equilibrium for a mean-field game of infinite-many symmetric players    with a time-inconsistent cost.

 \noindent{\bf Acknowledgements} The authors would like to thank  Professor Jiongmin Yong for his valuable discussions on the paper.

\end{document}